\documentclass[12pt]{amsart}
\usepackage[utf8]{inputenc}
\usepackage[english]{babel}
\usepackage{amsmath, amssymb, amsthm, amscd,color,comment}
\usepackage{cancel}
\usepackage{cite}
\usepackage{alltt}
\usepackage[dvipsnames]{xcolor}
\usepackage{array}
\usepackage[small,bf,labelsep=period]{caption}
\usepackage{mathtools}
\usepackage{hyperref}

\usepackage{pgfplots}
\pgfplotsset{compat=1.15}
\usepackage{mathrsfs}
\usetikzlibrary{arrows}

\usepackage{enumerate}
\newtheorem{theorem}{Theorem}[section]
\newtheorem{definition}[theorem]{Definition}

\newtheorem{remark}[theorem]{Remark}
\newtheorem{lemma}[theorem]{Lemma}
\newtheorem{corollary}[theorem]{Corollary}
\newtheorem{proposition}[theorem]{Proposition}

\DeclarePairedDelimiter\ceil{\lceil}{\rceil}
\DeclarePairedDelimiter\floor{\lfloor}{\rfloor}

\def\la{\lambda}
\def\a{\alpha}
\def\b{\beta}

\def\uu{{\bf u}}
\def\vv{{\bf v}}
\def\w{{\bf w}}
\def\x{{\bf x}}
\def\y{{\bf y}}
\def\z{{\bf z}}

\def\F{\mathbf F}
\def\P{\mathbf P}
\def\N{\mathbb N}
\def\R{\mathbb R}
\def\Z{\mathbb Z}
\def\cC{\mathcal C}
\def\cD{\mathcal D}
\def\cF{\mathcal F}
\def\cH{\mathcal H}
\def\cI{\mathcal I}
\def\cJ{\mathcal J}
\def\cL{\mathcal L}
\def\cl{\mathcal l}
\def\cP{\mathcal P}
\def\cM{\mathcal M}
\def\cO{\mathcal O}
\def\cK{\mathcal K}
\def\cG{\mathcal G}
\def\cW{\mathcal W}
\def\cX{\mathcal X}
\def\cY{\mathcal Y}
\def\l{\ell}
\def\a{\alpha}
\def\e{\epsilon}
\def\fqss{\mathbb F_{q^6}}
\def\fqs{\mathbb F_{q^2}}
\def\fqc{\mathbb F_{q^3}}
\def\fqn{\mathbb F_{q^n}}
\def\fqsn{{\mathbb F}_{q^{2n}}}
\def\fq{\mathbb F_q}
\def\fp{\mathbf F_p}
\def\fix{{\rm Fix}}
\def\sq{\sqrt{q}}
\def\aut{{\rm Aut}}
\def\gal{{\rm Gal}}
\def\supp{{\rm Supp}}
\def\div{{\rm div}}
\def\Div{{\rm Div}}
\def\dim{{\rm dim}}
\def\deg{{\rm deg}}
\def\ord{{\rm ord}}
\def\lub{{\rm lub}}
\def\glb{{\rm glb}}
\def\supp{{\rm supp}}
\def\res{{\rm res}}
\def\frx{{\mathbf Fr}_{\mathcal X}}
\def\a{\alpha}
\def\char{\mbox{\rm Char}}
\newcommand{\Mod}[1]{\ (\mathrm{mod}\ #1)}

\newcommand{\lu}[1]{{{\color{red}#1}}}
\newcommand{\alo}[1]{{{\color{purple}#1}}}
\newcommand{\er}[1]{{{\color{blue}#1}}}

\newcommand{\al}{\alpha}
\newcommand{\be}{\beta}
\newcommand{\g}{\gamma}

\begin{document}

\title[Weierstrass Semigroups, pure gaps and Codes on Function Fields]{Weierstrass Semigroups, pure gaps and Codes on Function Fields}



\thanks{The first author was partially supported by FAPEMIG. The second author was partially supported by FAPERJ/RJ-Brazil 201.650/2021. The third author thanks FAPERJ 260003/001703/2021 - APQ1,
CNPQ PQ 302727/2019-1  and
CAPES MATH AMSUD 88881.647739/2021-01 for the partial support.}

\author{Alonso S. Castellanos, Erik A. R. Mendoza, and Luciane Quoos}

\address{Faculdade de Matemática, Universidade Federal de Uberlândia, Campus Santa Mônica, CEP 38400-902, Uberlândia, Brazil}
\email{alonso.castellanos@ufu.br}

\address{Instituto de Matemática, Universidade Federal do Rio de Janeiro, Cidade Universitária,
	CEP 21941-909, Rio de Janeiro, Brazil}
\email{erik@im.ufrj.br}

\address{Instituto de Matemática, Universidade Federal do Rio de Janeiro, Cidade Universitária,
	CEP 21941-909, Rio de Janeiro, Brazil}
\email{luciane@im.ufrj.br}

\begin{abstract} 
For an arbitrary function field, from the knowledge of the minimal generating set of the Weierstrass semigroup at two rational places, the set of pure gaps is characterized. Furthermore, we determine the Weierstrass semigroup at one and two totally ramified places in a Kummer extension defined by the affine equation $y^{m}=\prod_{i=1}^{r} (x-\a_i)^{\lambda_i}$ over $K$, the algebraic closure of $\mathbb{F}_q$, where $\al_1, \dots, \al_r\in K$ are pairwise distinct elements, $1\leq \lambda_i < m$, and $\gcd(m, \sum_{i=1}^{r}\la_i)=1$. We apply these results to construct algebraic geometry codes over certain function fields with many rational places. For one-point codes we obtain families of codes with exact parameters.
\end{abstract}

\maketitle

\section{Introduction}
Algebraic function fields over a finite field with many rational places with respect to their genera have attracted a lot of attention in the context of error-correcting codes in the last decades. The so-called {\it algebraic geometry codes} (AG codes), introduced by Goppa \cite{G1977}, relies on a function field in one variable over a finite field $\fq$ with $q$ elements.

The Goppa's construction of linear codes over a function field $F/\fq$ of genus $g$ is described as follows. Let $D=P_1+\cdots +P_N$ and $G$ be divisors in $F/\fq$, where $P_1,\ldots, P_N$ are pairwise distinct rational places in $F$ and  such that $\supp(D)\;\cap\; \supp(G)=\emptyset$. Associated to the divisors $D$ and $G$ we have the linear algebraic geometry code
$$
C_{\mathcal{L}}(D,G)=\{(f(P_1),\dots,f(P_N))\mid f\in \cL(G)\}\subseteq \mathbb{F}_q^N,
$$
where $
\cL(G)=\{z\in F\mid (z)+G\geq 0\}\cup \{0\}
$ is the Riemann-Roch vector space associated to the divisor $G$. A code is said to be a one or two-point AG code if the divisor $G$ has support in one or two points respectively. The parameters: length equal to $N$, dimension $k$, and minimum distance $d$ of an AG code depend on the number of rational places, the genus $g$ of the function field, and the dimension of certain Riemann-Roch spaces. For the number $N(F)$ of rational places over $\fq$ on the function field $F/\fq$ we have the Hasse-Weil bound
\begin{equation}\label{maximal}
N(F) \leq q+1+2g\sqrt{q}.
\end{equation}
This is a deep result due to Hasse for elliptic curves, and due to A. Weil for general curves. We say that a function field (or the algebraic curve associated to it) is {\it maximal} if it achieves the Hasse-Weil bound. In particular, maximal function fields just exist over $\fqs$. The use of maximal function fields is indeed one of the main ingredients in the literature to construct AG codes with good parameters. Many examples of explicit maximal curves are described in \cite{AG2004, BM2018, CT2016, G2004, MQ2022} and references therein. It is worth to point out that several of the maximal curves admit a plane realization in the projective space $\mathbb{P}^2(\fq)$ as a Kummer extension, which is the central object of this work. 

Let $K$ be the algebraic closure of $\fq$, a finite field with characteristic $p$. A {\it Kummer extension} is defined by an affine equation 
\begin{equation}\label{curveX1}
\cX: \quad y^{m}:=f(x)=\prod_{i=1}^{r} (x-\a_i)^{\lambda_i}, \quad \la_i\in \N, \quad  1\leq \la_i< m \quad \text{and} \quad p\nmid m,
\end{equation}
where $\al_1, \dots, \al_r\in K$ are pairwise distinct elements, $\la_0:=\sum_{i=1}^{r}\la_i$, and $\gcd(m, \la_0)=1$. 
The investigation of Kummer extensions has attracted attention in recent years, see \cite{CMQ2016, MSW2021, GN2022, YH2018, YH2017, HY2018} for codes and semigroups, see \cite{CHT2016} for towers of function fields, and see \cite{GOR2019, CQT2022} for sequences over finite fields with high non-linear complexity.

An important object on the study of the theory of AG codes is related with the local structure of a function field, the Weierstrass semigroup at one or two rational places $P_1, P_2$, defined by
$$
H(P_1)=\{s\in \N\mid (z)_{\infty}=sP_1\text{ for some }z\in F\}
$$
and
$$
H(P_1, P_2)=\{(s_1, s_2)\in \N^2\mid (z)_{\infty}=s_1P_1+s_2P_2\text{ for some }z\in F\},
$$ 
respectively. The complement sets $G(P_1):=\N \setminus H(P_1)$ and $G(P_1, P_2):=\N^2 \setminus H(P_1, P_2)$ are called the set of gaps.

Weierstrass semigroups have been used to analyze the minimum distance, redundancy, and construction of AG codes with good parameters, see e.g. \cite{PT1999, CF2005, DK2011, G2001, G2004}. Explicit constructions of AG codes (in one and two points) over maximal curves can be found in \cite{DK2011} for the Hermitian curve, in \cite{CMQ2016,ST2014} for a generalization of the Hermitian curve, in \cite{CT2016} for codes over the $GK$ curve, and in \cite{LV2022} for the $BM$ curve. Many of the obtained constructions have examples of codes with good parameters. 

In 2001, Homma and Kim \cite{HK2001} investigated two-point codes over the Hermitian curve and introduced the very nice concept of {\it pure gaps} which turned out to be very useful for the improvement of the minimum distance of an AG code, see Theorem \ref{TorresCarvalho}. These ideas were generalized to many rational points by Carvalho and Torres, see \cite{CF2005}, and applied in recent publications such as \cite{BD2022} and \cite{HY2018}.

Kummer extensions as in (\ref{curveX1}) where all the multiplicities $\la_1=\la_2=\dots=\la_r$ are equal have been an object of interest concerning the theory of Weierstrass semigroups and codes, see \cite{HK2001, MSW2021} and \cite{YH2018}. In this case, for two totally ramified places in the Kummer extension, the Weierstrass semigroup and the minimal generating set was completely determined, see \cite{ CMQ2016, BQZ2018} and \cite{YH2017}.

The study of Kummer extensions as in (\ref{curveX1}) where not all the multiplicities $\la_1, \la_2, \dots, \la_r$ are  equal is a challenging problem and was just recently explored. In this case, for a totally ramified place $P$ in the Kummer extension, the authors in \cite{ABQ2019} provided an arithmetical criterion to determine if a positive integer is an element of the gap set $G(P)$. In \cite{M2022}, Mendoza explicitly describes the Weierstrass semigroup and the gap set at the only place at infinity.

In this work, we explore  Kummer extensions as described in (\ref{curveX1}). We provide an explicit description of the gap set at any totally ramified place and determine the minimal generating set at two totally ramified places. We apply the obtained results to construct one-point AG codes with exact parameters. In particular, we obtain a family of AG codes with Singleton defect $\delta=N-k-d=2$, see Remark \ref{singleton2}.  

For two rational places in an arbitrary function field, we present a characterization of the pure gap set in terms of the minimal generating set of the Weierstrass semigroup, see Proposition \ref{Gammasettopuregaps}. This characterization was very helpful for the applications. As a consequence, we determine pure gaps in two rational places over a maximal curve and construct codes with good parameters from Theorem \ref{TorresCarvalho}, see Table \ref{table2}.

We organize the paper as follows. Section 2 contains general results from the function field theory, Weierstrass semigroups and basic facts related to AG codes. In Section 3, we characterize the pure gap set $G_0(P_1, P_2)$ at two rational places $P_1$ and $P_2$ from the minimal generating set $\Gamma(P_1, P_2)$ for any function field (see Proposition \ref{Gammasettopuregaps}). 
In Section 4, we provide an explicit description of the gap set $G(P)$ at any totally ramified place $P$ in a Kummer extension as in (\ref{curveX1})(see Propositions \ref{prop1}, \ref{prop2} and \ref{prop3}). We also compute the minimal generating set $\Gamma(P_1, P_2)$ at two totally ramified places $P_1$ and $P_2$ in a Kummer extension (see Proposition \ref{propmgs1} and \ref{propmgs2}). In Section 5, we apply the results in the previous section to construct one-point AG codes. More specifically, we construct one-point AG codes over a  general family of Kummer extensions (see Theorem \ref{AGcodeQneqQ_inf}). In the same section, AG codes over particular function fields with many rational places are constructed given in all cases the exact value of their parameters (see Corollaries \ref{propc1}, \ref{propc2}, \ref{propc3} and Proposition \ref{propc4}). In Section 6, using the characterization of pure gaps in two rational places given in Proposition \ref{Gammasettopuregaps}, we compute pure gaps at two rational places over a certain maximal function field and construct two-point AG codes (see Proposition \ref{puregaps}, and Propositions \ref{CodesinQinfQ} and \ref{CodesinQ0Q1} respectively). Finally, in Section 7 we compare the relative parameters of the two-point AG codes obtained in Section 6 with the parameters of one-point AG codes.

\section{Preliminaries and notation}

Throughout this article, we let $q$ be the power of a prime $p$, $\fq$ be the finite field with $q$ elements, and $K$ be the algebraic closure of $\fq$. For $a, b$ integers, we denote by $(a, b)$ the greatest common divisor of $a$ and $b$. For $c\in \R$ a real number, we denote by $\floor*{c}$, $\ceil*{c}$ and $\{c\}$ the floor, ceiling and fractional part functions of $c$ respectively. We also let $\N=\{0,1, \dots \}$ be the set of natural numbers.

\subsection{Function Fields and Weierstrass semigroups}

Let $F/K$ be a function field of one variable of genus $g=g(F)$. We denote by $\mathcal P_F$ the set of places in $F$, by $\Omega_F$ the space of differentials forms in $F$, by $\nu_{P}$ the discrete valuation of $F/K$ associated to the place $P\in \cP_F$, and by $\Div (F)$ the free abelian group generated by the places in $F$. An element in $\Div (F)$ is called a divisor. For a function $z \in F$ we let $(z)_F, (z)_\infty$ and $(z)_0$ stand for the principal, pole and zero divisors of the function $z$ in $F$ respectively. 

Given a divisor $G\in \Div(F)$ of $F/K$, we have the following two vector spaces associate to $G$, the Riemann-Roch space
$$
\cL(G)=\{z\in F\mid (z)_F+G\geq 0\}\cup \{0\}
$$
  with dimension $\ell(G)$ as vector space over $K$, 
  and the space of  differentials given by
 $$\quad  \Omega(G)=\{\omega\in \Omega_F \mid (\omega)_F\geq G\}\cup \{0\}.$$
 
Now we introduce the notion of Weierstrass semigroup that plays an important role in the study of codes. 
For a place $P\in \cP_{F}$, the {\it Weierstrass semigroup} at $P$  is defined by
$$
H(P)=\{s\in \N \mid (z)_{\infty}=sP\text{ for some }z\in F\}.
$$
We say that a non-negative integer $s$ is a non-gap at $P$ if $s\in H(P)$. An element in the complement set $G(P):= \N \setminus H(P)$ is called a gap  at $P$. For a function field $F/K$ of genus $g>0$, the number of gaps is always finite, in fact $\#G(P)=g$. 

The Weierstrass semigroup $H(P)$ at one place admits a generalization for two places. Let $P_1$ and $P_2$ be distinct places in $F$. We define the Weierstrass semigroup associated to  $P_1, P_2$ by
$$
H(P_1, P_2)=\{(s_1, s_2)\in \N^2\mid (z)_{\infty}=s_1P_1+s_2P_2\text{ for some }z\in F\}.
$$
Analogously as in the case of one place, the elements of the set $G(P_1, P_2):=\N^2\setminus H(P_1, P_2)$ are called gaps at $P_1, P_2$.  Gaps can be characterized using Riemann-Roch spaces, that is, a pair $(s_1, s_2)\in \N^2$ is a gap at $P_1, P_2$ if and only if 
\begin{equation*}
\ell\left(s_1P_1+s_2P_2\right)=\ell\left(s_1P_1+s_2P_2-P_j\right)\text{ for some }j\in \{1, 2\}.
\end{equation*}

The set of gaps at two places $P_1, P_2$ can be obtained from the gaps at $P_1$ and $P_2$ as it follows.
Suppose that  $G(P_1)=\{\be_1<\be_2<\dots < \be_g\}$ and $G(P_2)=\{\gamma_1<\gamma_1<\dots < \gamma_g\}$. For each $i$, we let $n_{\be_i}= \min\{\gamma \in \N \mid (\be_i, \gamma)\in H(P_1, P_2)\}$. From \cite[Lemma 2.6]{K1994}, we have the equality $\{n_{\be} \mid \be\in G(P_1)\} = G(P_2)$, and therefore there exists a permutation $\sigma$ of the set $\{1, 2,\dots, g\}$ such that $n_{\be_i}= \gamma_{\sigma (i)}$. The graph of the bijective map between $G(P_1)$ and $G(P_2)$ defining the permutation $\sigma$ is the set $\Gamma (P_1, P_2)=\{(\be_i, \gamma_{\sigma (i)}) \mid i=1, \dots, g\}$. The following lemma characterizes the set $\Gamma(P_1, P_2)$.

\begin{lemma}\cite[Lemma 2]{H1996}\label{lemma1}
Let $\Gamma$ be a subset of
$(G(P_1)\times G(P_2)) \cap H(P_1, P_2)$. If there exists a permutation $\tau$ of $\{1, 2,\dots, g\}$ such that $\Gamma= \{(\al_i, \be_{\tau(i)})\mid i=1, \dots, g\}$,
then $\Gamma=\Gamma(P_1, P_2)$.
\end{lemma}

For $\x=(\be_1, \gamma_1)$ and $\y=(\be_2, \gamma_2)$, the least upper bound of $\x$ and $\y$ is defined as $\lub(\x, \y) =(\max\{\be_1, \be_2\}, \max\{\gamma_1, \gamma_2\})$. The following result shows that it is enough to determine $\Gamma(P_1, P_2)$ to compute the Weierstrass semigroup $H(P_1, P_2)$.

\begin{lemma}\cite[Lemma 2.2]{K1994}\label{genset}
Let $P_1$ and $P_2$ be two distinct places in $F$. Then 
$$
H(P_1, P_2)=\{\lub(\x, \y)\mid \x, \y\in \Gamma(P_1, P_2)\cup (H(P_1)\times \{0\})\cup (\{0\}\times H(P_2))\}.
$$
\end{lemma}

In this sense, the set $\Gamma(P_1, P_2)$ is called the {\it minimal generating set} of the Weierstrass semigroup $H(P_1, P_2)$. This set was computed in \cite{CT2016, G2004} for some places in families of function fields and used to provide codes with good parameters.

The next lemma will be an important tool in the computation of pure gaps (see Definition \ref{defpuregap}).

\begin{lemma}\cite[Noether's Reduction Lemma]{F1969}\label{fulton}
Let $D$ be a divisor, $P\in \cP_{F}$ and let $W$ be a canonical divisor. If $\ell(D)>0$ and $\ell(W-D-P)\neq \ell(W-D)$, then $\ell(D+P)=\ell(D)$. 
\end{lemma}

\subsection{Algebraic Geometry Codes}

For a function field $F/\fq$ with full constant field $\fq$, we say that a place $P$ is rational if it has degree one. 
In \cite{S2009}, the Goppa's construction of linear codes over a function field $F/\fq$ of genus $g$ is described as follows. Let $P_1,\ldots, P_N$ be pairwise distinct rational places in $F$ and $D:=P_1+\cdots +P_N$.  Consider  other  divisor  $G$ of $F$ such that $\supp(D)\;\cap\; \supp(G)=\emptyset$. Associated to the divisors $D$ and $G$ we have the linear algebraic geometry code $C_{\cL}(D,G)$  and the differential algebraic geometry code $C_{\Omega}(D,G)$ defined as 
$$
C_{\mathcal{L}}(D,G)=\{(f(P_1),\dots,f(P_N))\mid f\in \cL(G)\}\subseteq \mathbb{F}_q^N\;
$$
and
$$
C_{\Omega}(D,G)=\{(\res_{P_1}(\omega),\dots, \res_{P_N}(\omega))\mid \omega\in \Omega(G-D)\}\subseteq \mathbb{F}_q^N.
$$
The parameters of these codes are: $N$ is the length of the code, $k$ its dimension over $\fq$, and $d$ its minimum (hamming) distance. We say that the code is an $[N, k, d]$-code (AG code).  These codes are dual to each other, that is, $C_{\cL}(D, G)^{\perp}=C_{\Omega}(D, G)$. In what follows we have the classical lower bounds for the minimum distance of the linear and differential codes.

\begin{proposition}\cite[Corollary 2.2.3 and Theorem 2.2.7]{S2009}\label{goppabound}
Given the AG codes $C_{\cL}(D, G)$ and $C_{\Omega}(D, G)$ with parameters $[N, k, d]$ and $[N, k_{\Omega}, d_{\Omega}]$ respectively, we have that if $2g-2<\deg(G)<N$, then
$$
k=\deg(G)+1-g\quad \text{and}\quad d\geq N-\deg(G),
$$
and 
$$k_{\Omega}=N-\deg(G)-1+g\quad \text{and}\quad d_{\Omega}\geq \deg(G)-(2g-2).
$$ 
\end{proposition}

The well-known Singleton bound on a linear $[N, k, d]$-code establishes that $k+d \leq N+1$. The Singleton defect of the code is defined by $\delta=N+1-k-d \geq 0$ and can be used to measures how good is the code, that is, the smaller is the Singleton defect better is the code.

Now, we present a result that can be used to improve the lower bound for the minimum distance of an AG code.

\begin{theorem}\cite[Theorem 3]{GKL1993}\label{garcia}
Suppose that $\g-t, \g-t+1, \dots, \g-1, \g$ is a sequence of $t+1$ consecutive gaps at a rational place $Q$. Let $G=\g Q$ and $D=P_1 +P_2+ \cdots +P_N$, where $P_i$ is a rational place  not in the support of $G$ for each $i=1, \dots, N$. If the code $C_\cL(D, G)$ has positive dimension, then its minimum distance $d$ satisfies $$ d \geq N-\deg(G)+t+1.$$
\end{theorem}

\section{Pure Gaps in Function Fields}
In this section, we characterize the set of  pure gaps at two rational places over an arbitrary function field $F/\fq$. The notion of pure gaps at two places in a function field was introduced by Homma and Kim in \cite{HK2001}.

\begin{definition}\label{defpuregap}
The pair of natural numbers $(s_1, s_2)$ is a pure gap at the places $P_1, P_2$ if it satisfies
$$
\ell\left(s_1P_1+s_2P_2\right)=\ell\left(s_1P_1+s_2P_2-P_j\right)\text{ for all }j\in \{1, 2\}\,.
$$
We denote the pure gap set at $P_1, P_2$ by $G_0(P_1, P_2)$. 
\end{definition}

Equivalent, by \cite[Lemma 2.3]{HK2001}, we have that the pair $(s_1,s_2)$ is a pure gap if 
$$
\ell\left(s_1P_1+s_2P_2\right)=\ell\left((s_1-1)P_1+(s_2-1)P_2\right).
$$

Homma and Kim used this notion to provide a lower bound for the minimum distance of two-point differential AG codes.  

\begin{theorem}\cite[Theorem 3.3]{HK2001}\label{TorresCarvalho}
Let $P_1,\dots, P_N, Q_1,Q_2$ be pairwise distinct $\mathbb{F}_q$-rational places on the function field $F/\fq$ of genus $g$. Let $(\alpha_1,\alpha_2),(\beta_1,\beta_2)$ in $\mathbb{N}^2$ be such that $\al_i\leq \be_i$ for $i=1,2$. Suppose each pair $(\gamma_1,\gamma_2)$ with $\al_i\leq \gamma_i\leq \be_i$ for $i=1,2$ is a pure gap at $Q_1, Q_2$. Consider the divisors  $D=P_1+\cdots+P_N$ and $G=\sum_{i=1}^2(\alpha_i+\beta_i-1)Q_i$. Then  the minimum distance $d$  of the code $C_\Omega(D,G)$ satisfies
$$d\geq \deg(G)-(2g-2)+\sum_{i=1}^{2}(\be_i-\al_i)+2.$$

\end{theorem}

In the following we show a way to characterize pure gaps at two rational places $P_1, P_2$ in $F$ from the minimal generating set $\Gamma(P_1,P_2)$. 
For this we need the following definition: given two pairs $\x=(\be_1, \gamma_1)$ and $\y=(\be_2, \gamma_2)$ in $\N^2$, the {\it greatest lower bound} of $\x$ and $\y$ is defined as $$\glb(\x, \y) =(\min\{\be_1, \be_2\}, \min\{\gamma_1, \gamma_2\}).$$ 
\begin{proposition}\label{Gammasettopuregaps}
Let $P_1$ and $P_2$ be two distinct rational places in the algebraic function field $F/\fq$. Then the set of pure gaps $G_0(P_1, P_2)$ at $P_1, P_2$ is given by
$$
G_0(P_1, P_2)=\{\glb(\x, \y)\mid \x, \y\in \Gamma(P_1, P_2)\} \setminus \Gamma(P_1, P_2).
$$
\end{proposition}
\begin{proof}
Let $G(P_1)=\{\be_1<\be_2<\cdots<\be_g\}$, $G(P_2)=\{\gamma_1<\gamma_2<\cdots<\gamma_g\}$ be the set of gaps at $P_1$ and $P_2$ respectively.  Then it exists  $\sigma$ a permutation of $\{1, \dots, g\}$ such that  $$\Gamma (P_1, P_2)=\{(\be_i, \gamma_{\sigma (i)}) \mid i=1, \dots, g\}.$$ From \cite[Theorem 2.1]{HK2001}, the set of pure gaps can be  characterizes as
$$
G_0(P_1, P_2)=\{(\be_i, \gamma_j)\mid i<\sigma^{-1}(j)\text{ and } j<\sigma(i)\}.
$$
Let $(\be_i, \gamma_j)$ be an element of $G_0(P_1, P_2)$. Then $\be_i\in G(P_1)$, $\gamma_j\in G(P_2)$, and from definition of the minimal generating set $\Gamma (P_1, P_2)=\{(\be_i, \gamma_{\sigma (i)}) \mid i=1, \dots, g\}$ we have that $(\be_i, \gamma_{\sigma(i)})$ and $(\be_{\sigma^{-1}(j)}, \gamma_j)$ are elements of $\Gamma(P_1, P_2)$. Since $i<\sigma^{-1}(j)$ and $j<\sigma(i)$, it follows that $(\be_i, \gamma_j)=\glb((\be_i, \gamma_{\sigma(i)}), (\be_{\sigma^{-1}(j)}, \gamma_j))$ and $(\be_i, \gamma_j) \not\in \Gamma(P_1, P_2)$.

On the other hand, let $\x=(\be_{k}, \gamma_{\sigma(k)})$ and $\y=(\be_{l}, \gamma_{\sigma(l)})$ be elements in $\Gamma(P_1, P_2)$ such that $\glb(\x, \y)\not\in \Gamma(P_1, P_2)$. Without loss of generality, suppose that $k<l$. If $\sigma(k)\leq\sigma(l)$, then $\glb(\x, \y)=\x\in \Gamma(P_1, P_2)$, a contradiction. Therefore $\sigma(k)>\sigma(l)$ and $\glb(\x, \y)=(\be_k, \gamma_{\sigma(l)})$, where $k<\sigma^{-1}(\sigma(l))$ and $\sigma(l)<\sigma(k)$, that is, $\glb(\x, \y)\in G_0(P_1, P_2)$.
\end{proof}

\section{The Weierstrass semigroup at one and two rational places over a Kummer extension}

Consider the curve $\cX$ defined by the affine equation 
\begin{equation}\label{curveX}
\cX: \quad y^{m}:=f(x)=\prod_{i=1}^{r} (x-\a_i)^{\lambda_i}, \quad \la_i\in \N, \quad  1\leq \la_i< m \quad \text{and} \quad p \nmid m,
\end{equation}
where $\al_1, \dots, \al_r\in K$ are pairwise distinct elements, $\la_0:=\sum_{i=1}^{r}\la_i$ and $(m, \la_0)=1$. 
Let $\cF=K(\cX)$ be its function field. Then $\cF/K(x)$ is a Kummer extension with exactly one place at infinity. By \cite[Proposition 3.7.3]{S2009}, its genus $g(\cX)$ is given by
$$
g(\cX)=\frac{m(r-1)+1-\sum_{i=1}^{r}(m, \la_i)}{2}.
$$ 
For $i=1, \dots, r$, let $P_{i}$ and $P_{\infty}$ be the places in $\cP_{K(x)}$ corresponding to the zero of $x-\al_i$ and the pole of $x$ respectively. If $(m, \la_i)=1$ we denote by $Q_i$ the only place in $\cF$ over $P_i$ and by $Q_\infty$ the only place over $P_\infty$.

Suppose for a moment that all the multiplicities $\la:=\la_1=\la_2=\dots=\la_r$ are the same, that is, $y^m =f(x)^\la$ where $f(x)$ is a separable polynomial over $K$ and $(m, r\lambda)=1$. In this case, since $(m, \la)=1$, the function field of $\cX$ is isomorphic to $K(x, z)$ where $z^m=f(x)$ for $z=y^{r_1}f^{r_2}$ and $r_1, r_2$ integers satisfying  $r_1\la+r_2m=1$. So, without loss of generality, if the multiplicities of the roots of $f(x)$ are the same and co-prime with $m$, we let all of them equal to 1.

In this section, we compute the Weierstrass semigroup at any totally ramified place in the extension $\cF/K(x)$. Furthermore, for two totally ramified places in the extension $\cF/K(x)$, we determine the minimal generating set of the corresponding Weierstrass semigroup.

To describe the minimal generating set at two totally ramified places in the extension $\cF/K(x)$, we start by providing another description of the gap set $G(Q_{\infty})$ given in \cite[Proposition 4.1]{M2022}, and computing the gap set at a totally ramified place $Q_{\ell}\in \cP_{\cF}$, where $Q_{\ell}\neq Q_{\infty}$.
\begin{proposition}\label{prop1}
The gap set at the only place at infinity $Q_{\infty}$ of $\cF$ is given by 
$$
G(Q_{\infty})=\left\{mj-i\la_0\mid 1\leq i\leq m-1, \, \ceil*{\frac{i\la_0}{m}}\leq j \leq \sum_{\ell=1}^{r}\ceil*{\frac{i\la_\ell}{m}}-1\right\}. 
$$ 
\end{proposition}
\begin{proof} Define the set
$$
G:=\left\{mj-i\la_0\mid 1\leq i\leq m-1, \, \ceil*{\frac{i\la_0}{m}}\leq j \leq \sum_{\ell=1}^{r}\ceil*{\frac{i\la_\ell}{m}}-1\right\}. 
$$ 
For $mj-i\la_0\in G$, let $t$ be the unique element in $\{0, \dots, m-1\}$ such that $mj-i\la_0 - t\la_0\equiv 0\mod{m}$. Since $(m, \la_0)=1$ we get  $-i \equiv t\mod{m}$, so  $\{\frac{t\la_\ell}{m}\}=\{\frac{-i\la_\ell}{m}\}$ for $1\leq \ell \leq r$. Then
$$
\sum_{\ell=1}^{r}\left\{\frac{t\la_\ell}{m}\right\}=\sum_{\ell=1}^{r}\left\{\frac{-i\la_\ell}{m}\right\}=\sum_{\ell=1}^{r}\left(-\frac{i\la_\ell}{m}-\floor*{-\frac{i\la_\ell}{m}}\right)=-\frac{i\la_0}{m}+\sum_{\ell=1}^{r}\ceil*{\frac{i\la_\ell}{m}}.
$$
From the definition of $G$ we have
$$
\sum_{\ell=1}^{r}\left\{\frac{t\la_\ell}{m}\right\}=-\frac{i\la_0}{m}+\sum_{\ell=1}^{r}\ceil*{\frac{i\la_\ell}{m}} > -\ceil*{\frac{i\la_0}{m}}+\sum_{\ell=1}^{r}\ceil*{\frac{i\la_\ell}{m}} \geq j-\floor*{\frac{i\la_0}{m}} = \ceil*{\frac{mj-i\la_0}{m}}.
$$
Applying \cite[Corollary 3.6]{ABQ2019}, we conclude that $mj-i\la_0\in G(Q_{\infty})$. This yields  $G\subseteq G(Q_{\infty})$. 

On the other hand, since $\#\{1 \leq s \leq m-1: m \text{ divides }s\la_{\ell}\}=(m, \la_\ell)-1$ for $1\leq \ell \leq r$, we have that
\begin{align*}
\#G&=\sum_{i=1}^{m-1}\left[\left(\sum_{\ell=1}^{r}\ceil*{\frac{i\la_\ell}{m}}\right)-\ceil*{\frac{i\la_0}{m}}\right]\\
&=\sum_{\ell=1}^{r}\sum_{i=1}^{m-1}\ceil*{\frac{i\la_\ell}{m}}-\sum_{i=1}^{m-1}\ceil*{\frac{i\la_0}{m}}\\
&=\sum_{\ell=1}^{r}\left(m-(m, \la_\ell)+\sum_{i=1}^{m-1}\floor*{\frac{i\la_\ell}{m}}\right)-\frac{(m-1)(\la_0+1)}{2}\\
&=\sum_{\ell=1}^{r}\frac{(m-1)(\la_\ell-1)-(m, \la_\ell)-1+2m}{2}-\frac{(m-1)(\la_0+1)}{2}\\
&=\frac{m(r-1)+1-\sum_{\ell=1}^{r}(m, \la_\ell)}{2}\\
&=g(\cX),
\end{align*}
and this concludes the desired result $G(Q_{\infty})=G$.
\end{proof}

In the next proposition we compute the gap set at $Q_{\ell}$, a totally ramified place in the extension $\cF/K(x)$ different from $Q_{\infty}$.
\begin{proposition}\label{prop2}
Suppose that $(m, \la_\ell)=1$ for some $1\leq \ell\leq r$ and let $1\leq \la \leq m-1$ be the inverse of $\la_{\ell}$ modulo $m$. Let $Q_{\ell}\in \cP_{\cF}$ be the unique extension of $P_{\ell}$. Then 
$$
G(Q_{\ell})=\left\{i+mj\mid 1\leq i\leq m-1, \, 0\leq j \leq \left(\sum_{k=1}^{r}\ceil*{\frac{i\la\la_k}{m}}\right)-\ceil*{\frac{i\la\la_0}{m}}-1\right\}. 
$$ 
\end{proposition}
\begin{proof} 
We start by defining the set
$$
G:=\left\{i+mj\mid 1\leq i\leq m-1, \, 0\leq j \leq \left(\sum_{k=1}^{r}\ceil*{\frac{i\la\la_k}{m}}\right)-\ceil*{\frac{i\la\la_0}{m}}-1\right\}. 
$$ 
For $i+mj\in G$, let $t$ be the unique element in $\{0, \dots, m-1\}$ such that $i+mj + t\la_{\ell}\equiv 0\mod{m}$. So $-i\la \equiv t \mod{m}$ and we get  $\{\frac{t\la_k}{m}\}=\{\frac{-i \la\la_k}{m}\}$ for $1\leq k \leq r$.  Now we have
$$
\sum_{k=1}^{r}\left\{\frac{t\la_k}{m}\right\}=\sum_{k=1}^{r}\left\{\frac{-i\la\la_k}{m}\right\}=\sum_{k=1}^{r}\left(-\frac{i\la\la_k}{m}-\floor*{\frac{-i\la\la_k}{m}}\right)=-\frac{i\la\la_0}{m}+\sum_{k=1}^{r}\ceil*{\frac{i\la\la_k}{m}}.
$$
Since $(m, \la)=(m, \la_0)=1$, it follows that
$$
\sum_{k=1}^{r}\left\{\frac{t\la_k}{m}\right\}=-\frac{i\la\la_0}{m}+\sum_{k=1}^{r}\ceil*{\frac{i\la\la_k}{m}} > -\ceil*{\frac{i\la\la_0}{m}}+\sum_{k=1}^{r}\ceil*{\frac{i\la\la_k}{m}} \geq j+1 = \ceil*{\frac{i+mj}{m}}.
$$
Thus, from \cite[Corollary 3.6]{ABQ2019}, it follows that $G\subseteq G(Q_{\ell})$. Similarly to the proof of  Proposition \ref{prop1}, it can be proved that $\# G=g(\cX)$ and therefore $G(Q_{\ell})=G$.
\end{proof}

Furthermore, for the case $\la_1=\la_2=\dots=\la_r=1$, we  give another description of the gap set $G(Q)$ at a totally ramified place $Q\neq Q_{\infty}$ in $\cF/K(x)$. With this new characterization it will be easier to identify all consecutive sequences of gaps.
\begin{proposition}\label{prop3} 
Suppose that $\la_1=\la_2=\dots=\la_r=1$ and let $Q$ be a totally ramified place in $\cF/K(x), Q\ne Q_{\infty}$. Then
$$
G(Q)=\left\{mj-i\mid 1\leq j \leq r-1, \, \floor*{\frac{jm}{r}}+1\leq i \leq m-1\right\}.
$$
\end{proposition}
\begin{proof}
First of all, define the set 
$$
G:=\left\{mj-i\mid 1\leq j \leq r-1, \, \floor*{\frac{jm}{r}}+1\leq i \leq m-1\right\}\,,
$$
and note that
$$
\# G=\sum_{j=1}^{r-1}\left(m-1-\floor*{\frac{jm}{r}}\right)=(m-1)(r-1)-\sum_{j=1}^{r-1}\floor*{\frac{jm}{r}}=\frac{(m-1)(r-1)}{2}=g(\cX).
$$ 
On the other hand, for each $mj-i\in G$, let $t$ be the unique element in $\{0, \dots, m-1\}$ such that $mj-i+t\equiv 0 \mod{m}$, then $t=i$. Moreover, since $mj/r< \floor*{mj/r}+1\leq i$, we have $j<ir/m$ and
$$
\sum_{\ell=1}^{r}\left\{\frac{t\la_\ell}{m}\right\}=\frac{ir}{m}>j=\ceil*{\frac{mj-i}{m}}.
$$
From \cite[Corollary 3.6]{ABQ2019}, we obtain that $G\subseteq G(Q)$ and therefore $G=G(Q)$.
\end{proof}

Now we describe the minimal generating set for the Weierstrass semigroup at two totally ramified places in $\cF$ with the same multiplicity.

\begin{proposition}\label{propmgs1} 
Suppose that $\la_{\ell_1}=\la_{\ell_2}$ and $(m, \la_{\ell_1})=1$ for some $1\leq \ell_1, \ell_2\leq r$. Let $\la$ be the inverse of $\la_{\ell_1}$ modulo $m$, and $Q_{\ell_s}\in \cP_{\cF}$ be the unique extension of $P_{\ell_s}$ for $s=1, 2$. Then
\begin{align*}
\Gamma(Q_{\ell_1}, Q_{\ell_2})=\Bigg\{(i+mj_1, i+mj_2)\in \N^2 &\mid 1\leq i \leq m-1, \, j_1\geq 0,\, j_2\geq 0,\\
&\quad j_1+j_2=\left(\sum_{k=1}^{r}\ceil*{\frac{i\la\la_k}{m}}\right)-\ceil*{\frac{i\la\la_0}{m}}-1 \Bigg\}.
\end{align*}
\end{proposition}
\begin{proof}
Without loss of generality, suppose that $\ell_1=1$ and $\ell_2=2$. Define the set
\begin{align*}
\Gamma:=\Bigg\{(i+mj_1, i+mj_2)\in \N^2 &\mid 1\leq i \leq m-1, \, j_1\geq 0,\, j_2\geq 0,\\
&\quad j_1+j_2=\left(\sum_{k=1}^{r}\ceil*{\frac{i\la\la_k}{m}}\right)-\ceil*{\frac{i\la\la_0}{m}}-1 \Bigg\}.
\end{align*}
We are going to prove that $\Gamma=\Gamma(Q_1, Q_2)$. For $k=1, \dots, r$, from \cite[Proposition 3.7.3]{S2009}, we have the principal divisors
\begin{equation}\label{div}
\begin{array}{l}
(x-\al_k)_{\cF}=\dfrac{m}{(m, \la_k)}\displaystyle\sum_{{ Q\in \cP_{\cF}, \, Q|P_{k}}}Q-mQ_{\infty}\quad \text{and}\\
(y)_{\cF}=\displaystyle\sum_{k=1}^{r}\frac{\la_k}{(m, \la_k)}\sum_{{ Q\in \cP_{\cF}, \, Q|P_{k}}}Q-\la_0Q_{\infty}.
\end{array} 
\end{equation}
On the other hand, since $(m, \la_1)=1$, then there exists $\be \in \Z$ such that $\be m + \la \la_1 =1$. Given the tuple  $(i+mj_1, i+mj_2)$ in $ \Gamma$, after some computations, we have the following divisor
\begin{align*}
&\left((x-\al_1)^{-(j_1+\be i)} (x-\al_2)^{-(j_2+\be i)} y^{-i\la}\prod_{k=3}^{r}(x-\al_k)^{\ceil*{\frac{i\la\la_k}{m}}} \right)_{\cF}\\
&=\sum_{k=3}^{r}\frac{m\ceil*{\frac{i\la\la_k}{m}}-i\la\la_k}{(m, \la_k)}\sum_{ Q\in \cP_{\cF}, \, Q|P_{k}}Q
+\left(i\la\la_0-m\floor*{\frac{i\la\la_0}{m}}\right) Q_{\infty}\\
&-(i+mj_1)Q_{1}-(i+mj_2)Q_{2},
\end{align*}
\noindent
where $m\ceil*{i\la\la_k/m}-i\la\la_k$, $i\la\la_0-m\floor*{i\la\la_0/m}$, $i+mj_1$, and $i+mj_2$ are non-negative integers. This proves that $\Gamma\subseteq H(Q_1, Q_2)$. 

On the other hand, since $j_1, j_2\geq 0$ and $j_1+j_2=\left(\sum_{k=1}^{r}\ceil*{\frac{i\la\la_k}{m}}\right)-\ceil*{\frac{i\la\la_0}{m}}-1$, we have
$$
0\leq j_s\leq \left(\sum_{k=1}^{r}\ceil*{\frac{i\la\la_k}{m}}\right)-\ceil*{\frac{i\la\la_0}{m}}-1\quad \text{for } s=1, 2.$$
From Proposition \ref{prop2}, we conclude that $\Gamma\subseteq G(Q_1)\times G(Q_2)$. Therefore $\Gamma\subseteq (G(Q_1)\times G(Q_2))\cap H(Q_1, Q_2)$. 

Moreover, again from Proposition \ref{prop2}, the set $\Gamma$ can be seen as the graph of the bijective map $\theta: G(Q_1)\rightarrow G(Q_2)$ given by $\theta(i+mj_1)=i+mj_2$, which defines a permutation $\tau$ of the set $\{1, \dots, g(\cX)\}$. From Lemma \ref{lemma1}  we conclude that $\Gamma=\Gamma(Q_1, Q_2)$.
\end{proof}

In particular, if $\la_1= \la_2= \dots =\la_r$ in the Proposition \ref{propmgs1}, we obtain the description of the minimal generating set  given in \cite[Theorem 8]{YH2017}.

\begin{proposition}\label{propmgs2} 
Suppose that $\la_{\ell}=1$ for some $1\leq \ell\leq r$ and let $Q_{\ell}\in \cP_{\cF}$ be the unique extension of $P_{\ell}$. Then
\begin{align*}
\Gamma(Q_{\infty}, Q_{\ell})=\Bigg\{(mj_1-i\la_0, i+mj_2)\in \N^2 &\mid 1\leq i \leq m-1, \, j_1\geq \ceil*{\frac{i\la_0}{m}}, \, j_2\geq 0,\\
&\quad j_1+j_2=\sum_{k=1}^{r}\ceil*{\frac{i\la_k}{m}} -1 \Bigg\}.
\end{align*}
\end{proposition}
\begin{proof}
Without loss of generality, suppose that $\ell=1$ and define the set
\begin{align*}
\Gamma:=\Bigg\{(mj_1-i\la_0, i+mj_2)\in \N^2 &\mid 1\leq i \leq m-1, \, j_1\geq \ceil*{\frac{i\la_0}{m}}, \, j_2\geq 0,\\ & j_1+j_2=\sum_{k=1}^{r}\ceil*{\frac{i\la_k}{m}} -1 \Bigg\}.
\end{align*}
Since $j_1\geq \ceil*{\frac{i\la_0}{m}}$, $j_2\geq 0$, and $j_1+j_2=\sum_{k=1}^{r}\ceil*{\frac{i\la_k}{m}}-1$, we have
$$
\ceil*{\frac{i\la_0}{m}}\leq j_1 \leq \sum_{k=1}^{r}\ceil*{\frac{i\la_k}{m}}-1\quad \text{and}\quad 0\leq j_2\leq \left(\sum_{k=1}^{r}\ceil*{\frac{i\la_k}{m}}\right)-\ceil*{\frac{i\la_0}{m}}-1.
$$
From Propositions \ref{prop1} and \ref{prop2}, it follows that $\Gamma \subseteq G(Q_{\infty})\times G(Q_1)$. 

On the other hand, for $(mj_1-i\la_0, i+mj_2)\in \Gamma$ and from (\ref{div}), we have the following divisor
\begin{align*}
&\left((x-\al_1)^{-j_2}y^{-i}\prod_{k=2}^{r}(x-\al_k)^{\ceil*{\frac{i\la_k}{m}}}\right)_{\cF}\\
&=\sum_{k=2}^{r}\frac{m\ceil*{\frac{i\la_k}{m}}-i\la_k}{(m, \la_k)}\sum_{Q\in \cP_{\cF},\, Q|P_{k}}Q
-(mj_1-i\la_0) Q_{\infty}-(i+mj_2)Q_{1}
\end{align*}
and therefore $\Gamma\subseteq H(Q_1, Q_2)$. Thus, $\Gamma\subseteq (G(Q_{\infty})\times G(Q_1))\cap H(Q_{\infty}, Q_1)$. Similarly, as in the proof of Proposition \ref{propmgs1}, $\Gamma$ is the graph of the bijective map $\theta: G(Q_{\infty})\rightarrow G(Q_1)$ given by $\theta(mj_1-i\la_0)=i+mj_2$, which defines a permutation $\tau$ of the set $\{1, \dots, g(\cX)\}$ that satisfies the conditions of the Lemma \ref{lemma1}. It follows that $\Gamma=\Gamma(Q_{\infty}, Q_1)$.
\end{proof}

\section{One-point Codes}

In this section we construct one-point AG codes over Kummer extensions. We start by presenting a general construction of linear codes using the results obtained in the previous sections. As a consequence, we construct three families of one-point AG codes on function fields with many rational places and provide the exact value of their parameters.  

\begin{theorem}\label{AGcodeQneqQ_inf}
Let $\cX$ be the curve defined by $y^m=f(x)$, where $f(x) \in \fq[x]$ is a separable polynomial of degree $r\geq 3$. Let $Q\in \cP_{\fq(\cX)}$ be a totally ramified place in the extension $\fq(\cX)/\fq(x)$ such that $Q\neq Q_{\infty}$. For $a\in \{2, \dots, r-1\}$, define the divisors 
$$
G_a:=\left(am-\floor*{\frac{am}{r}}-1\right)Q\quad \text{and} \quad D:=\sum_{Q'\in \cX(\fq), Q'\neq Q}Q',
$$ 
where $\cX(\fq)$ is the set of $\fq$-rational places on the function field $\fq(\cX)$, and assume that $\deg(G_a)<N:=\deg (D)$. Then the linear AG code $C_{\cL}(D, G_a)$ has parameters
$$
\left[N, a+\sum_{i=1}^{a-1}\floor*{\frac{im}{r}}, d\geq N-m(a-1)\right].
$$
In addition, if $\# \{\gamma\in \fq \mid P_{\gamma} \in \cP_{\fq(x)} \text{ splits completely in }\fq(\cX)/\fq(x)\}\geq a-1$, then the minimum distance of the linear code $C_{\cL}(D, G_a)$ is exactly $d=N-m(a-1)$.
\end{theorem}
\begin{proof}
From Proposition \ref{prop3}, the gap set $G(Q)$ can be decomposed as the disjoint union of the following consecutive sequences of gaps
\begin{equation}\label{gapsequence}
(j-1)m+1, (j-1)m+2, \dots, (j-1)m+\left(m-\floor*{\frac{jm}{r}}-1\right)
\end{equation}
of length $m-\floor*{\frac{jm}{r}}-1$
for $1\leq j \leq r-1$. Given $a$ in $\{2, \dots, r\}$, from (\ref{gapsequence}), we deduce that
$$
\ell(G_a)=\ell\left(\left(am-\floor*{\frac{am}{r}}-1\right)Q\right)=a+\sum_{i=1}^{a-1}\floor*{\frac{im}{r}}.
$$
Thus, since $\deg(G_a)<\deg (D)$ and from Theorem \ref{garcia}, the AG code $C_{\cL}(D, G_a)$ has parameters
$$
\left[N, a+\sum_{i=1}^{a-1}\floor*{\frac{im}{r}}, d\geq N-m(a-1)\right].
$$
Now, let $S:=\{\gamma\in \fq\mid P_{\gamma} \in \cP_{\fq(x)} \text{ splits completely in the extension }\fq(\cX)/\fq(x)\}$ and suppose that $\#S\geq a-1$. For $ 2 \leq a \leq r-1$, consider the function 
$$
z:=(x-\beta)^{1-a}\prod_{i=1}^{a-1}(x-\gamma_i),
$$
where $\gamma_i\in S$ and $\beta\in \fq$ is such that $Q$ is the only place over $P_{\beta}\in\cP_{\fq(x)}$. Then $z$ is in $\cL((am-\floor*{am/r}-1)Q)$ and has exactly $m(a-1)$ distinct zeros. The weight of the corresponding codeword is $N-m(a-1)$ and the result follows.
\end{proof}

In the following, we apply Theorem \ref{AGcodeQneqQ_inf} to  construct one-point AG codes on function fields with many rational places.

Let $q=p^n$, $p$ prime and $n\geq 2$. In \cite{AG2004} Abd\'{o}n and Garcia showed that there exists a unique $\fqs$-maximal function field with genus $q(q/p-1)/2$ having a place such that $q/p$ is a non-gap at this place. We present a first family of codes over this function field.
\begin{corollary}\label{propc1}
Let $q=p^n$, $p$ prime, and $n\geq 2$ such that $3\leq q/p$. For each $2\leq a \leq q/p-1$, it exists an AG code over $\fqs$ with parameters 
$$
\left[\frac{q^3}{p}, a+\frac{pa(a-1)}{2}, \frac{q^3}{p}-(q+1)(a-1)\right].
$$
\end{corollary}
\begin{proof} Consider the $\fqs$-maximal function field of the curve defined by the affine equation 
$$cy^{q+1}=x^{q/p}+x^{q/p^2}+\cdots + x^p+x, \quad c^q+c=0 \text{ and } c\neq 0$$ 
with genus $g= q(q/p-1)/2$. From Theorem \ref{AGcodeQneqQ_inf} the result follows.
\end{proof}

In the following result we present one-point AG codes over a generalization of the Hermitian function field given by Garcia in \cite[Example 1.3]{G1992}.

\begin{corollary}\label{propc2}
For $q \geq 3$, $n$ an odd integer, and $2\leq a \leq q-1$, it exists an AG code over $\fqsn$ with parameters
$$
\left[q^{2n+1}, a+\frac{a(a-1)q^{n-1}}{2}, q^{2n+1}-(q^n+1)(a-1)\right].
$$
\end{corollary}
\begin{proof}
Consider the $\fqsn$-maximal function field of the curve defined by the equation
$$
y^{q^n+1}=x^q+x.
$$
This function field has genus $g=q^n(q-1)/2$ and note that when $n = 1$ we get the Hermitian function field. Using Theorem \ref{AGcodeQneqQ_inf}, we obtain one-point codes with the desired parameters.
\end{proof}

As a last application, we construct one-point AG codes over the function field of the Norm-Trace curve. This function field was studied in detail by Geil in \cite{G2003}.

\begin{corollary}\label{propc3}
For $q\geq 3, n \geq 2, $ and $2\leq a \leq q^{n-1}-1$, we obtain one-point AG codes over $\fqn $ with parameters
$$
\left[q^{2n-1}, \frac{a(a+1)}{2}+\sum_{i=1}^{a-1}\floor*{\frac{i(q^{n-1}-1)}{q^{n-1}(q-1)}}, q^{2n-1}-\frac{(a-1)(q^n-1)}{q-1}\right].
$$
\end{corollary}
\begin{proof} Let $n\geq 2$ be an integer. The Norm-Trace curve over $\fqn$ is defined by
$$
y^{\frac{q^n-1}{q-1}}=x^{q^{n-1}}+x^{q^{n-2}}+\cdots +x.
$$
Its function field has genus $g=\frac{q(q^{n-1}-1)^2}{2(q-1)}$ and $q^{2n-1}+1$ rational places over $\fqn$. 
The desired result follows immediately by applying Theorem \ref{AGcodeQneqQ_inf} to this function field.
\end{proof}

Now we provide a family of AG codes with exact parameters over the function field $\fqsn(\cY_{m})$, where $\cY_m$ is defined by the Equation $(\ref{curvecode})$.

Let $q\geq 4$ even, $n\geq 3$ be an odd integer,  and $m\geq 2$ be an integer such that $m$ divides $q^n+1$ and $q+1$ divides $m$. Consider the curve given by the affine equation 
\begin{equation}\label{curvecode}
\cY_{m}:\quad y^{m}=x(x+1)\left(\frac{x^{q-1}+1}{x+1}\right)^{q+1}.
\end{equation}
This curve is a subcover of the Beelen-Montanucci curve \cite{BM2018} and first appeared in \cite[Theorem 3.1]{MQ2022}. Its function field $\fqsn(\cY_m)$ is maximal over $\fqsn$, has genus $$g=\frac{m(q-1)-q^2+q+1}{2}$$ and the number of rational places is 
$$\#\cY_{m}(\fqsn)=q^{2n}-q^{n+2}+(m+1)q^{n+1}-(m-1)q^n+1.$$  
The only totally ramified places in the extension $\fqsn(\cY_m)/\fqsn(x)$ are the places $Q_\infty$ and $Q_\alpha$ that lie over the places $P_\infty$ and $P_{\al}$ in $\cP_{\fqsn(x)}$ for $\al \in \{0, 1\}$. For $\be$ in $\fq\setminus \{0, 1\}$, the place $P_{\be}\in \cP_{\fqsn(x)}$ has exactly $q+1$ rational places in $\fqsn(\cY_{m})$ over $P_{\be}$. 
From \cite[Theorem 4]{OT2014}, for each  $\gamma\in\fqsn\setminus \fq$ the place  $P_{\gamma}\in \cP_{\fqsn(x)}$ has exactly none or $m$ rational places in $\fqsn(\cY_{m})$ over $P_{\gamma}$.
Let $u$ be the number of elements $\gamma\in\fqsn\setminus \fq$ such that $P_{\gamma}\in \cP_{\fqsn(x)}$ has exactly $m$ rational places in $\fqsn(\cY_{m})$ over $P_{\gamma}$. Then 
\begin{equation*}
\#\cY_{m}(\fqsn)=3+(q+1)(q-2)+um
\end{equation*} 
and we conclude $u=(q^n-q^2+q)(q^n+1)/m+q^{n+1}-q^n$. 

For $q-1\leq k\leq u$ consider the divisors
$$
G:=kmQ_\infty
\quad 
\text{and}
\quad
D:=\sum_{Q\in \cY_m(\fqsn), \, Q\neq Q_{\infty}}Q,
$$
and the code $C_{\cL}(D,G)$.
This code has length 
$$N:=\deg(D)=q^{2n}-q^{n+2}+(m+1)q^{n+1}-(m-1)q^n$$ and its minimum distance is $d=N-km$. In fact, the function $z_k=\prod_{i=1}^k(x-\gamma_i)$ in $\cL(kmQ_\infty)$, where $\gamma_i\in \fqsn\setminus\fq$ is such that $P_{\gamma_i}$ splits completely in $\fqsn(\cY_{m})/\fqsn(x)$, has exactly $km$ distinct zeros and the weight of the corresponding codeword is $N-km$. Also, since $2g-2 < km < N$, the code has dimension $km+1-g$. We summarize this result in the next proposition.

\begin{proposition}\label{propc4}
Let $q\geq 4$ be even, $n \geq 3$ odd, and $m\geq 2$ be an integer such that $m$ divides $q^n+1$ and $q+1$ divides $m$. For $q-1\leq k\leq (q^n-q^2+q)(q^n+1)/m+q^{n+1}-q^n$, it exists a linear code over $\fqsn$ with parameters
\begin{equation*}\label{par2}
[N, km+1-((q-1)m-q^2+q+1)/2, N-km],
\end{equation*}
where $N=q^{2n}-q^{n+2}+(m+1)q^{n+1}-(m-1)q^n$.
\end{proposition}
\begin{remark}\label{singleton2}
In particular, for $m=q+1$ we have a code with parameters
$$[q^{2n}+q^{n+1}, k(q+1)+1-q/2, q^{2n}+q^{n+1}-k(q+1)]\;$$ 
and singleton defect $\delta=q/2$.
We notice that for  $q=4$ we get a code over $\fqsn$ with singleton defect $\delta=2$. 
\end{remark}

\section{Two-point Codes}\label{twopoint}

Previously, in Propositions \ref{propmgs1} and \ref{propmgs2} it was obtained a description of the minimal generating set at two totally ramified places in a Kummer extension. 
Now, benefiting from these descriptions and applying Proposition \ref{Gammasettopuregaps} and Theorem \ref{TorresCarvalho}, we construct two-point AG codes on the subcover of the Beelen-Montanucci curve described in (\ref{curvecode}). For two-point AG codes on the Beelen-Montanucci curve, see \cite{LV2022}.

For $q \geq 4$ even, consider the curve  in (\ref{curvecode}) for $m=q^n+1$, that is, 
\begin{equation}\label{curvecode1}
\cY_{q^n+1}:\quad y^{q^n+1}=x(x+1)\left(\frac{x^{q-1}+1}{x+1}\right)^{q+1},
\end{equation}
and let $\fqsn(x, y)$ be its function field.
Fix the place $Q_\infty$, and the other two totally ramified places of degree one $Q_0, Q_1$ in the extension $\fqsn(x, y)/\fqsn(x)$. By Propositions \ref{propmgs1} and \ref{propmgs2} we have the following descriptions of the minimal generating sets
\begin{align}\label{GammaQ0Q1}
\Gamma(Q_0,Q_1)=\Bigg\{(i+mj_1,i+mj_2) &\mid 1\leq i\leq m-1,\; j_1\geq 0,\; j_2\geq 0, \mbox{ and }\\
&\quad j_1+j_2=(q-2)\ceil*{ \frac{i}{s}}+1-\ceil*{\frac{i(q^2-q)}{m}} \Bigg\}\nonumber
\end{align}
and
\begin{align}\label{GammaQinfQ}
\Gamma(Q_\infty,Q)=
\Bigg\{(mj_1-i(q^2-q),i+mj_2) &\mid 1\leq i\leq m-1,\; j_1\geq \left\lceil\frac{i(q^2-q)}{m}\right\rceil,\; j_2\geq 0,\\
&\quad \mbox{and }j_1+j_2=(q-2)\left\lceil \frac{i}{s}\right\rceil+1 \Bigg\}, \nonumber
\end{align}
where $Q\in \{Q_0, Q_1\}$. 

Furthermore, we notice that the divisor 
$$W:=(2g-2)Q_{\infty}=(q^{n+1}-q^n-q^2+2q-2)Q_\infty$$ is a canonical divisor of the function field $\fqsn(x, y)$. In fact, from \cite[Theorem 4.4]{M2022}, the Weierstrass semigroup $H(Q_{\infty})$ is symmetric and therefore $2g-1\in G(Q_{\infty})$. From the Riemann-Roch Theorem, $\ell(uQ_{\infty})=u+1-g$ for $u\geq 2g-1$ and therefore
$$
\ell(W)=\ell((2g-2)Q_{\infty})=\ell((2g-1)Q_{\infty})=g.
$$
With the same notation above, we present pure gaps at two rational places on the function field of the curve $\cY_{q^n+1}$.

\begin{proposition}\label{puregaps}
Let $q\geq 4$ be even, $n\geq 3$ be an odd integer, $s=\frac{q^n+1}{q+1}$, and $Q \in \{Q_0, Q_1\} $ a totally ramified place in the extension $\fqsn(x,y)/\fqsn(x)$ as in (\ref{curvecode1}). Then 
\begin{enumerate}[i)]
\item For $0\leq a\leq s-2$ and $1\leq b\leq s-1-a$ we have 
$$((q^n+1)(q-1)-(s-a)(q^2-q), b) \in G_0(Q_\infty, Q).$$ 
\item  For  $1\leq a_i\leq q+1$ and $1\leq b_i \leq (q^2-q-2a_i)\frac{(q^{n-1}-1)}{q^2-1}$ for $i=1, 2$ we have
$$\left(\frac{a_2(q^n-2q^{n-1}+1)}{q-1}-b_2, \frac{a_1(q^n-2q^{n-1}+1)}{q-1}-b_1\right)\in G_0(Q_0, Q_1).$$
\item For
\begin{itemize}
\item $1\leq a_2\leq a_1\leq q+1,$ 
\item $0 \leq b_1 \leq 
\begin{cases}
2a_1\frac{(q^{n-1}-1)}{q^2-1}, \text{ if } 1 \leq a_1 \leq \frac{q}{2},\\
\frac{q^{n-1}-q}{q-1}, \text{ if } \frac{q}{2}+1 \leq a_1 \leq q+1, 
\end{cases}$ 
\item $ 1\leq b_2\leq (q^2-q-2a_2)\frac{(q^{n-1}-1)}{q^2-1},$\\
\end{itemize}
 the pairs
$$
\left(\frac{a_1(q^n-2q^{n-1}+1)}{q-1}+b_1, \frac{a_2(q^n-2q^{n-1}+1)}{q-1}-b_2\right)
$$
and 
$$
\left(\frac{a_2(q^n-2q^{n-1}+1)}{q-1}-b_2, \frac{a_1(q^n-2q^{n-1}+1)}{q-1}+b_1\right)
$$
are pure gaps at $Q_0, Q_1$.
\end{enumerate}
\end{proposition}
\begin{proof}   
For the first item, choosing  $i=s-a, j_1=q-1$, and $j_2=0$ in Equation (\ref{GammaQinfQ}), we have that for $0\leq a \leq s-1$
$$\uu_a:=((q^n+1)(q-1)-(s-a)(q^2-q),s-a)\in \Gamma(Q_\infty,Q).$$
Thus, from Proposition \ref{Gammasettopuregaps} we obtain that
$$
\glb (\uu_a, \uu_{s-b})=((q^n+1)(q-1)-(s-a)(q^2-q), b)\in G_{0}(Q_{\infty}, Q)
$$  
for $0\leq a\leq s-2$ and $1\leq b \leq s-1-a$.

Now we are going to divide the proof of the second and third items into two steps. For simplicity let $M:=\frac{q^n-2q^{n-1}+1}{q-1}$. 

{\bf Claim 1: }For $1\leq a\leq q+1$ and $0\leq b \leq \tilde{M}:=\min\left\{2a\frac{(q^{n-1}-1)}{q^2-1}, \ceil*{\frac{q^n+1-2a}{q^2-q}}-1\right\}$, we have
$$
\w_{a, b}:=\left(aM+b, aM+b\right)\in \Gamma(Q_0, Q_1).
$$ 
{\it Proof of Claim 1: } Given $a$ and $b$, choose the values $j_1=0$, $j_2=0$, and $i=aM+b$ in the description of the set $\Gamma(Q_0, Q_1)$ given in (\ref{GammaQ0Q1}). Then we are left to prove that $1 \leq i \leq q^n $ and $(q-2)\ceil*{ \frac{i}{s}}+1-\ceil*{\frac{i(q^2-q)}{q^n+1}}=0$. Note that
\begin{align*}
\ceil*{\frac{i}{s}}&=\ceil*{\frac{a(q^n-2q^{n-1}+1)(q+1)}{(q-1)(q^n+1)}+\frac{b(q+1)}{q^n+1}}\\
&=\ceil*{\frac{a(q^{n+1}-q^n-2q^{n-1}+q+1)}{q^{n+1}-q^n+q-1}+\frac{b(q^2-1)}{(q^n+1)(q-1)}}\\
&=\ceil*{a-\frac{2a(q^{n-1}-1)}{(q^n+1)(q-1)}+\frac{b(q^2-1)}{(q^n+1)(q-1)}}\\
&=a+\ceil*{\frac{b(q^2-1)-2a(q^{n-1}-1)}{(q^n+1)(q-1)}}=a,
\end{align*}
since $0\leq b\leq \frac{2a(q^{n-1}-1)}{q^2-1}$.
Analogously, 
\begin{align*}
\ceil*{\frac{i(q^2-q)}{q^n+1}}&=\ceil*{\frac{a(q^n-2q^{n-1}+1)(q^2-q)}{(q-1)(q^n+1)}+\frac{b(q^2-q)}{q^n+1}}\\
&=\ceil*{\frac{a(q^{n+1}-2q^n+q)}{q^n+1}+\frac{b(q^2-q)}{q^n+1}}\\
&=a(q-2)+\ceil*{\frac{2a+b(q^2-q)}{q^n+1}}=a(q-2)+1,
\end{align*}
since $0\leq b\leq \ceil*{\frac{q^n+1-2a}{q^2-q}}-1.$

We now compute the minimum $\tilde{M}$. For $n \geq 3$ we are going to show that
\begin{align*}
\tilde{M}=
\begin{cases}
2a\frac{(q^{n-1}-1)}{q^2-1}, &\text{ if } 1 \leq a \leq \frac{q}{2},\\
\frac{q^{n-1}-q}{q-1}, &\text{ if } \frac{q}{2}+1 \leq a \leq q+1.
\end{cases}
\end{align*}
In fact, at first let $ 1 \leq a \leq \frac{q}{2}$, then 
$$\ceil*{\frac{q^n+1-2a}{q^2-q}}-1 \geq \frac{q^n+1-q}{q^2-q}-1=\frac{q^n+1-q^2}{q^2-q}\geq \frac{q^n-q}{q^2-1} \geq 2a\frac{(q^{n-1}-1)}{q^2-1}.$$
On the other hand, if $\frac{q}{2} + 1 \leq a \leq q+1$, then 
$$2a\frac{(q^{n-1}-1)}{q^2-1} \geq (q+2)\frac{q^{n-1}-1}{q^2-1}\geq \frac{q^n-q-1}{q^2-q}\geq \ceil*{\frac{q^n+1-2a}{q^2-q}}-1=\sum_{i=1}^{n-2} q^i=\frac{q^{n-1}-q}{q-1}.$$
Then it is easy to conclude that $i=aM+b \leq q^n$.
The claim follows.

{\bf Claim 2}: For $1\leq a\leq q+1$ and $1\leq b \leq (q^2-q-2a)\frac{(q^{n-1}-1)}{q^2-1}$, the pair
$$
\uu_{a, b}:=\left(aM-b, aM-b\right)\in G_0(Q_0, Q_1).
$$
{\it Proof of Claim 2: } From definition of pure gap we must prove that $\cL(E)=\cL(E-Q_i),$ where $E=\left(aM-b\right)(Q_0+Q_1)$, for $i=0, 1$. The cases $i=0,1$ are analogous, and then we prove only for $i=0$.

Given $a$ and $b$, define the function
$$
\mu:=y^{aM-b-1}(x+1)\prod_{\al\in \fq\setminus\{0, 1\}}(x-\al)^{1-a}\in \fqsn(x, y).
$$
After some computations, we obtain that the principal divisor of $\mu$ in $\fqsn(x, y)$  is 
\begin{align*}
(\mu)_{\fqsn(x, y)}&=\left(aM-b-1\right)Q_0+ \left(aM+q^n-b\right)Q_1\\
&\quad +\left(\frac{(q^{n-1}-1)(q^2-q-2a)}{q^2-1}-b\right)\sum_{Q|P_{\al}, \, \al\in \fq\setminus \{0, 1\}}Q\\
&\quad -\left(q^{n+1}-q^n+q-1+2a-(b+1)(q^2-q)\right)Q_{\infty}.
\end{align*}
We obtain that $\mu\in \cL(W-E+Q_0)\setminus \cL(W-E)$, where $W=(q^{n+1}-q^n-q^2+2q-2)Q_\infty$ is a canonical divisor of the function field $\fqsn(x, y)$. From Lemma \ref{fulton}, we conclude that $\cL(E)=\cL(E-Q_0)$. This completes the proof of the second claim. 

Since $\uu_{a, b}$ is a pure gap at $Q_0, Q_1$, from (\ref{GammaQ0Q1}) and Proposition \ref{Gammasettopuregaps}, it follows that there exists positive integers $j_1, j_2$ such that  $\uu_{a, b}=\glb(\vv^1_{a, b}, \vv^2_{a, b})$ for the  pairs
\begin{align*}
&\vv^1_{a, b}:=(aM-b+j_1(q^n+1), aM-b) \text{ and} \\
&\vv^2_{a, b}:=(aM-b, aM-b+j_2(q^n+1)) 
\end{align*}
in $\Gamma(Q_0, Q_1).$

Thus, for $1\leq a_k\leq q+1$, $1\leq b_k \leq (q^2-q-2a_k)\frac{(q^{n-1}-1)}{q^2-1}$ for $k=1, 2$ we have 
$$
\glb(\vv^1_{a_1, b_1}, \vv^2_{a_2, b_2})=\left(a_2M-b_2, a_1M-b_1\right)\in G_0(Q_0, Q_1).
$$
On the other hand, for $1\leq a_2\leq a_1\leq q+1$, $0\leq b_1\leq \min\left\{\frac{2a_1(q^{n-1}-1)}{q^2-1}, \ceil*{\frac{q^n+1-2a_1}{q^2-q}}-1\right\}$, and $1\leq b_2\leq (q^2-q-2a_2)\frac{(q^{n-1}-1)}{q^2-1}$ we obtain that
$$
\glb(\w_{a_1, b_1}, \vv^1_{a_2, b_2})=\left(a_1M+b_1, a_2M-b_2\right)\in G_0(Q_0, Q_1)
$$
and 
$$
\glb(\w_{a_1, b_1}, \vv^2_{a_2, b_2})=\left(a_2M-b_2, a_1M+b_1\right)\in G_0(Q_0, Q_1).
$$
\end{proof}

Note that Proposition \ref{Gammasettopuregaps} was the fundamental key for determining the pure gaps given in the previous proposition. To see a geometric interpretation of Proposition \ref{Gammasettopuregaps}, in Figure \ref{figure1} we determine the pure gap set $G_0(Q_0, Q_1)$ from the minimal generating set $\Gamma(Q_0, Q_1)$ on the curve $\cY_{q^n+1}$ defined in (\ref{curvecode1}) for the case $q=4$ and $n=3$.

\definecolor{ffqqqq}{rgb}{1.,0.,0.}
\begin{figure}[h!]
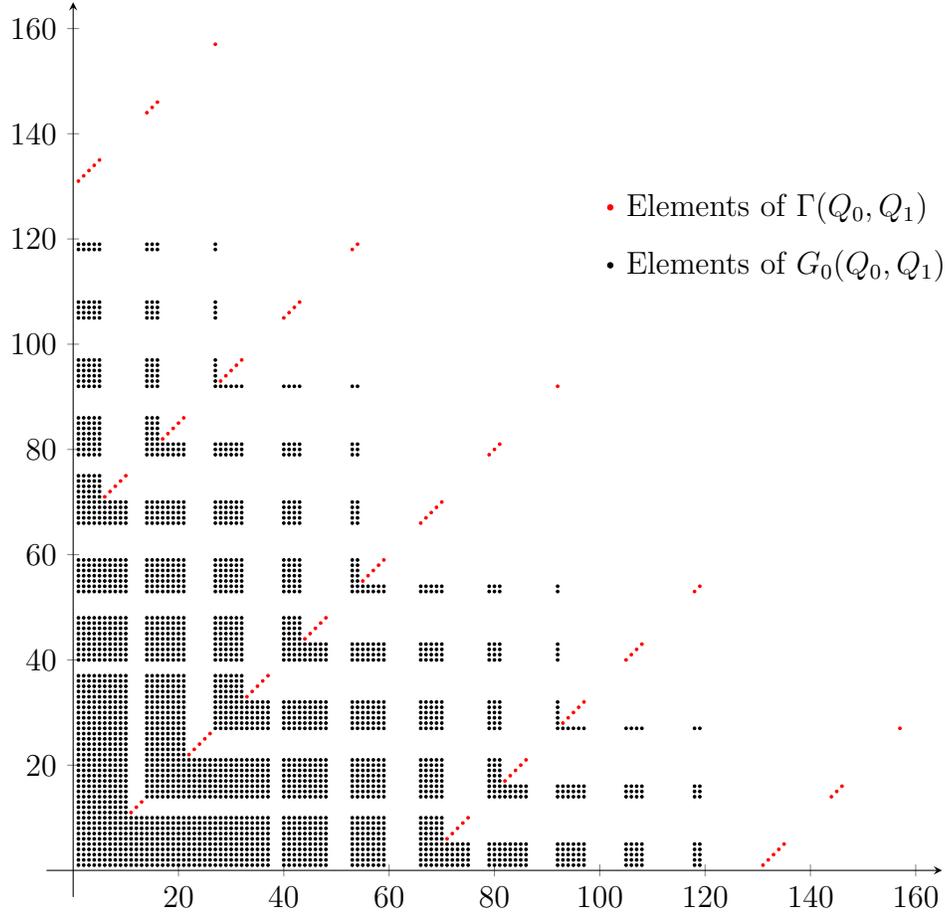


\caption{Determination of $G_0(Q_0, Q_1)$ from $\Gamma(Q_0, Q_1)$ for $q=4$ and $n=3$.}\label{figure1}
\end{figure}

To finish this section, we construct two-point AG codes using Proposition \ref{puregaps} and Theorem \ref{TorresCarvalho}.

\begin{proposition}\label{CodesinQinfQ}
Let $q\geq 4$ be even and $n\geq 3$ be an odd integer. Then
for $$\floor*{\frac{q^{n+2}-2q^{n+1}-q^3+q^2+1}{2q^3-3q-1}}+1\leq a\leq \frac{q^n-2q-1}{q+1}$$ it exists a $[N, k, d]$-code over $\fqsn$ with parameters
\begin{align*}
&N=q^{2n+1}-q^{n+2}+2q^{n+1}-1, \\
&k=q^{2n+1}-q^{n+2}+\frac{5q^{n+2}+q^n-q^3+q^2-2q+2}{2(q+1)}-a(2q^2-2q-1), \text{ and}\\
&d\geq 2a(q^2-q-1)-\frac{q^2(q^n-2q^{n-1}-q^{n-2}-q+1)}{q+1}.
\end{align*}
\end{proposition}
\begin{proof}
Consider $Q_{\infty}$ the only place at infinity of the function field $\fqsn(x, y)$ and $Q$  any other totally ramified place of degree one in the extension $\fqsn(x, y)/\fqsn(x)$. Let $s=\frac{q^n+1}{q+1}$. From item $i)$ in Proposition \ref{puregaps}, we have the following  pure gaps at $Q_{\infty}, Q$ for the values of $b=1$ and $b=s-1-a$
\begin{equation}\label{firstpuregaps}
((q^n+1)(q-1)-(s-a)(q^2-q), 1)\text{ and }((q^n+1)(q-1)-(s-a)(q^2-q), s-1-a).
\end{equation}
Define the divisors
$$
G:=(2(q^n+1)(q-1)-2(s-a)(q^2-q)-1)Q_{\infty}+(s-1-a)Q,
$$
and
$$
D:=\sum_{Q'\in \cY_{q^n+1}(\fqsn)\setminus \{Q_{\infty}, Q\}}Q'.
$$
 Then $\deg (G)<\deg (D)=q^{2n+1}-q^{n+2}+2q^{n+1}-1$ and 
\begin{align*}
\deg(G)&=2(q^{n+1}-q^n+q-2)-(s-a)(2q^2-2q-1)\\
&>2(q^{n+1}-q^n+q-2)-\left(s-\frac{q^{n+2}-2q^{n+1}-q^3+q^2+1}{(q+1)(2q^2-2q-1)}\right)(2q^2-2q-1)\\
&=2(q^{n+1}-q^n+q-2)-s(2q^2-2q-1)+\frac{q^{n+2}-2q^{n+1}-q^3+q^2+1}{q+1}\\
&=\frac{q^{n+2}-q^n-q^3+q^2-2}{q+1}\\
&=2g(\cY_{q^n+1})-2.
\end{align*}
So the AG code $C_{\Omega}(D, G)$ has dimension
\begin{align*}
k_{\Omega}&= \deg(D)+g-1-\deg (G)\\
&=q^{2n+1}-q^{n+2}+\frac{5q^{n+2}+q^n-q^3+q^2-2q+2}{2(q+1)}-a(2q^2-2q-1).
\end{align*}
Since $((q^n+1)(q-1)-(s-a)(q^2-q), b)$ is a pure gap at $Q_\infty, Q$ for any $1\leq b \leq s-1-a$, from Proposition \ref{puregaps} and Theorem \ref{TorresCarvalho}, it follows that the AG code $C_{\Omega}(D, G)$ has minimum distance
\begin{align*}
d_{\Omega}&\geq \deg (G)-(2g-2)+s-a\\
&=2a(q^2-q-1)-\frac{q^2(q^n-2q^{n-1}-q^{n-2}-q+1)}{q+1}.
\end{align*}
\end{proof}

\begin{proposition}\label{CodesinQ0Q1}
Let $n\geq 3$ be an odd integer and $u=2^n$. For $\frac{4u^2+9}{5}\leq c_1 \leq \frac{11u^2+4}{12}$ and $\frac{4u^2+9}{5}\leq c_2 \leq  \frac{5u^2+4}{6}$, there exists an AG code over $\mathbb{F}_{u^4}$ with parameters
$$
\left[4u^4-8u^2-1, 4u^4-\frac{33u^2}{4}-5-c_1-c_2, d\geq \frac{u^2}{2}+12\right].
$$
\end{proposition}

\begin{proof}
Let $M=\frac{u^2+2}{6}$, $R=\frac{u^2-4}{60}$, and $b=\frac{u^2-16}{12}$. For $q=4$ and $a_1=a_2=q+1$ in items $ii)$ and $iii)$ of Proposition \ref{puregaps}, we deduce that the elements of the set
$$
\left\{(n_1, n_2)\in \N^2 \mid 5M-2R\leq n_1\leq 5M+b, \, 5M-2R\leq n_2\leq 5M-1\right\}
$$
are pure gaps at $Q_0, Q_1$, where $Q_0, Q_1$ are the totally ramified places in $\mathbb{F}_{u^4}(\cY_{u^2+1})/\mathbb{F}_{u^4}(x)$ distinct from $Q_{\infty}$. Consider the pairs
\begin{equation}\label{secondpuregaps}
(c_1, c_2)\quad\text{and}\quad(5M+b, 5M-1),
\end{equation}
where $5M-2R\leq c_1\leq 5M+b$ and $5M-2R\leq c_2\leq 5M-1$, and the divisors
$$
G:=(5M+b+c_1-1)Q_0+(5M+c_2-2)Q_1
$$
and
$$
D:=\sum_{Q'\in \cY_{u^2+1}(\mathbb{F}_{u^4})\setminus \{Q_0, Q_1\}}Q'.
$$
The pairs given in (\ref{secondpuregaps}) are pure gaps at $Q_0, Q_1$ and satisfy the conditions of Theorem \ref{TorresCarvalho}. Furthermore, since $\deg(G)<\deg(D)=4u^4-8u^2-1$ and
\begin{align*}
\deg (G)=10M+c_1+c_2+b-3
\geq 20M-4R+b-3
>2g(\cY_{u^2+1})-2,
\end{align*}
we conclude that the AG code $C_{\Omega}(D, G)$ over $\mathbb{F}_{u^4}$ has dimension
\begin{align*}
k= \deg(D)+g-1-\deg (G)
=4u^4-\frac{33u^2}{4}-5-c_1-c_2,
\end{align*}
and minimum distance satisfying 
\begin{align*}
d\geq \deg (G)-(2g-2)+1+b+10M-c_1-c_2
=\frac{u^2}{2}+12.
\end{align*}
\end{proof}

\section{Some tables of codes}
In this section, we compare the relative parameters of two-point AG codes over the function field $\fqsn(\cY_{q^n+1})$ obtained in Propositions \ref{CodesinQinfQ} and \ref{CodesinQ0Q1}, with the relative parameters of one-point AG codes over the same function field obtained using the order bound.

Let $P$ be a rational place in a function field $F/\fq$ and set 
$$
H(P)=\{\rho_1 := 0 < \rho_2 <\dots\}
$$
the Weierstrass semigroup at $P$. The {\it Feng-Rao designed minimum distance} or the {\it order bound} of $H(P)$ is defined by the function $d_{ORD}: \N\rightarrow \N$ given by
$$
d_{ORD}(\ell):=\min\{\nu_m\mid m\geq \ell\},
$$
where $\nu_{\ell}:=\#\{(i, j)\in \N^2 \mid \rho_i + \rho_j =\rho_{\ell+1}\}$. In general, we have that $d_{ORD}(\ell)\geq \ell+1-g$ and the equality holds if $\rho_{\ell}\geq 4g-1$, see \cite[Theorem 5.24]{HVP1998}. For one-point differential AG codes, we can use the order bound for obtain a lower bound for the minimum distance.   
\begin{theorem}\cite[Theorem 4.13]{HVP1998}\label{orderbound}
Consider the one-point code $C_{\ell}:=C_{\cL}(P_1+\dots+P_N, \rho_{\ell}P)^{\perp}$, where $P, P_1, \dots, P_N$ are distinct $\fq$-rational places in $F$ and $\rho_{\ell}\in H(P)$ is such that $N>\rho_{\ell}$. Then $C_{\ell}$ is an $[N, N-\ell, \geq d_{ORD}(\ell)]$-code over $\fq$. 
\end{theorem}

For $Q$ a totally ramified place in the extension $\fqsn(\cY_{q^n+1})/\fqsn(x)$ such that $Q\neq Q_{\infty}$, we can use the description of the gap set $G(Q)$ given in Proposition \ref{prop2} and the order bound described in Theorem \ref{orderbound} to obtain one-point AG codes over the function field $\fqsn(\cY_{q^n+1})$. In Table \ref{table1}, using the package NumericalSgps \cite{DGM2022} of the software GAP \cite{GAP4}, we present parameters of 
one-point AG codes over $\fqsn(\cY_{q^n+1})$ for $q=4$ and $n=3$. These codes have length $N=15872$ and are defined over $\mathbb{F}_{2^{12}}$.

On the other hand, in Table \ref{table2} we present two-point AG codes over $\mathbb{F}_{2^{12}}$ of length $N=15871$  obtained from Proposition \ref{CodesinQinfQ} (for $q=4$ and $n=3$) and Proposition \ref{CodesinQ0Q1} (for $n=3$). In all cases we obtain better relative parameters with respect to the one-point AG codes obtained on Table \ref{table1}.

\vspace{1cm}
\hspace{-1cm}
\begin{minipage}[c]{0.5\textwidth}
\centering
\begin{tabular}
{
|>{\centering\arraybackslash}p{1cm}|>{\centering\arraybackslash}p{0.8cm} |>{\centering\arraybackslash}p{0.9cm} |
}
\hline
$k$ & $\rho_{\ell}$ & $d\geq$ \\ \hline 
$ 15620 $ & $ 343 $ & $ 160 $ \\ \hline
$ 15643 $ & $ 320 $ & $ 137 $ \\ \hline
$ 15666 $ & $ 297 $ & $ 114 $ \\ \hline
$ 15689 $ & $ 274 $ & $ 91 $ \\ \hline
$ 15712 $ & $ 251 $ & $ 72 $ \\ \hline
$ 15735 $ & $ 228 $ & $ 50 $ \\ \hline
$ 15741 $ & $ 222 $ & $ 39 $ \\ \hline
$ 15742 $ & $ 221 $ & $ 39 $ \\ \hline
$ 15743 $ & $ 220 $ & $ 39 $ \\ \hline
$ 15744 $ & $ 219 $ & $ 39 $ \\ \hline
$ 15745 $ & $ 218 $ & $ 39 $ \\ \hline
$ 15758 $ & $ 205 $ & $ 30 $ \\ \hline
\end{tabular}
\captionof{table}{One-point AG codes over $\mathbb{F}_{2^{12}}$ of length $N=15872$.}\label{table1}
\end{minipage}
\hfill
\begin{minipage}[c]{0.5\textwidth}
\centering
\begin{tabular}
{
|>{\centering\arraybackslash}p{1cm}|>{\centering\arraybackslash}p{0.8cm} |>{\centering\arraybackslash}p{5cm} |
}
\hline
$k$ & $d\geq $ & Reference \\ \hline 
$ 15620 $ & $ 162 $ & Prop. \ref{CodesinQinfQ}: $a=11$ \\ \hline
$ 15643 $ & $ 140 $ & Pro. \ref{CodesinQinfQ}: $a=10$ \\ \hline
$ 15666 $ & $ 118 $ & Prop. \ref{CodesinQinfQ}: $a=9$\\ \hline
$ 15689 $ & $ 96 $ & Prop. \ref{CodesinQinfQ}: $a=8$\\ \hline
$ 15712 $ & $ 74 $ & Prop. \ref{CodesinQinfQ}: $a=7$\\ \hline
$ 15735 $ & $ 52 $ & Prop. \ref{CodesinQinfQ}: $a=6$\\ \hline
$ 15741 $ & $ 44 $ & Prop. \ref{CodesinQ0Q1}: $c_1=56, \, c_2=54$ \\ \hline
$ 15742 $ & $ 44 $ & Prop. \ref{CodesinQ0Q1}: $c_1=55, \, c_2=54$ \\ \hline
$ 15743 $ & $ 44 $ & Prop. \ref{CodesinQ0Q1}: $c_1=c_2=54$\\ \hline
$ 15744 $ & $ 44 $ & Prop. \ref{CodesinQ0Q1}: $c_1=53, \, c_2=54$\\ \hline
$ 15745 $ & $ 44 $ & Prop. \ref{CodesinQ0Q1}: $c_1=c_2=53$\\ \hline
$ 15758 $ & $ 30 $ & Prop. \ref{CodesinQinfQ}: $a=5$ \\ \hline
\end{tabular}
\captionof{table}{Two-point AG codes over $\mathbb{F}_{2^{12}}$ of length $N=15871$.}\label{table2}
\end{minipage}

\bibliographystyle{abbrv}



\end{document}